\documentclass[12pt,A4,twoside]{amsart}

\usepackage{fullpage}
\usepackage{stmaryrd}

\usepackage[latin1]{inputenc}

\usepackage[ngerman,english]{babel}

\usepackage{etex}
\usepackage{tikz}
\usepackage{tikz-cd}

\usepackage{amsmath}
\usepackage{amsthm}
\usepackage[mathscr]{eucal}
\usepackage{mathbbol}
\usepackage{amssymb}         
\usepackage{epsfig}          
\usepackage{longtable}       
\usepackage{multicol}        

\usepackage{graphicx}               
\usepackage{colortbl}
\usepackage{fancybox}
\usepackage[arrow, matrix, curve]{xy}

\usepackage{mathtools}
\usepackage[perpage]{footmisc}

\mathtoolsset{showonlyrefs} \mathtoolsset{showmanualtags}

\makeatletter
\def\swappedhead#1#2#3{%
  \thmnumber{\@upn{\the\thm@headfont#2\@ifnotempty{#1}{~}}}%
  \thmname{#1}%
  \thmnote{ {\the\thm@notefont(#3)}}}
\makeatother

\swapnumbers
\theoremstyle{plain}
\newtheorem{theorem}{Theorem}[section]
\newtheorem{lemma}[theorem]{Lemma}
\newtheorem{corollary}[theorem]{Corollary}
\newtheorem{proposition}[theorem]{Proposition}

\theoremstyle{definition}
\newtheorem{definition}[theorem]{Definition}

\theoremstyle{remark}
\newtheorem{example}[theorem]{Example}
\newtheorem{remark}[theorem]{Remark}
\newtheorem{note}[theorem]{Note}
\newtheorem{observation}[theorem]{Observation}

\newcommand{\Span}{\operatorname{span}}

\newcommand{\betrag}[1]{\left|#1\right|}
\newcommand{\gklammer}[1]{\left\{#1\right\}}

\newcommand{\BIGOP}[1]{\mathop{\mathchoice%
{\raise-0.22em\hbox{\huge $#1$}}%
{\raise-0.05em\hbox{\Large $#1$}}{\hbox{\large $#1$}}{#1}}}
\newcommand{\bigtimes}{\BIGOP{\times}}
\newcommand{\BIGboxplus}{\mathop{\mathchoice%
{\raise-0.35em\hbox{\huge $\boxplus$}}%
{\raise-0.15em\hbox{\Large $\boxplus$}}{\hbox{\large
    $\boxplus$}}{\boxplus}}}

\usepackage{hyperref}

\newenvironment{env}[2]{\begin{#1}#2\end{#1}}{}
    \newcommand{\beq}[1]{\begin{env}{equation}{#1}}
    \newcommand{\beqn}[1]{\begin{env}{equation*}{#1}}
    \newcommand{\bal}[1]{\begin{env}{align}{#1}}
    \newcommand{\baln}[1]{\begin{env}{align*}{#1}}
    \newcommand{\bga}[1]{\begin{env}{gather}{#1}}
    \newcommand{\bgan}[1]{\begin{env}{gather*}{#1}}
    \newcommand{\bflal}[1]{\begin{env}{flalign}{#1}}
    \newcommand{\bflaln}[1]{\begin{env}{flalign*}{#1}}
    \newcommand{\bmu}[1]{\begin{env}{multline}{#1}}
    \newcommand{\bmun}[1]{\begin{env}{multline*}{#1}}
    \newcommand{\bsp}[1]{\begin{env}{split}{#1}}

    \newcommand{\eeq}{\end{env}}
    \newcommand{\eeqn}{\end{env}}
    \newcommand{\eal}{\end{env}}
    \newcommand{\ealn}{\end{env}}
    \newcommand{\ega}{\end{env}}
    \newcommand{\egan}{\end{env}}
    \newcommand{\eflal}{\end{env}}
    \newcommand{\eflaln}{\end{env}}
    \newcommand{\emu}{\end{env}}
    \newcommand{\emun}{\end{env}}
    \newcommand{\esp}{\end{env}}

\newcommand{\lf}{\vspace{2ex}}

\newcommand{\SQ}[1]{\vspace{-4ex}\begin{quotation}\renewcommand{\baselinestretch}{1}\small\noindent\ignorespaces#1\end{quotation}}
\renewcommand{\bf}[1]{\textbf{#1}}
\renewcommand{\it}[1]{\textit{#1}}

\renewcommand{\sf}[1]{\textsf{#1}}

\renewcommand{\tt}[1]{\texttt{#1}}
\newcommand{\hl}[1]{\bf{\it{#1}}}

\newcommand{\msf}[1]{\text{\small $\sf{#1}$}}

\newcommand{\cmc}[1]{\mathcal{#1}}
\newcommand{\eus}[1]{\mathscr{#1}}

\newcommand{\bb}[1]{\mathbb{#1}}

\newcommand{\mscriptsize}[1]{{\setlength{\arraycolsep}{.3ex}\text{\scriptsize$#1$}}}

\newcommand{\nbd}[1]{$#1$\nobreakdash--}

\newcommand{\ul}[1]{\underline{#1}}

\newcommand{\wh}[1]{\widehat{#1}}

\newcommand{\bfam}[1]{\bigl(#1\bigr)}

\newcommand{\CB}[1]{\{#1\}}
\newcommand{\bCB}[1]{\bigl\{#1\bigr\}}
\newcommand{\BCB}[1]{\Bigl\{#1\Bigr\}}
\newcommand{\SB}[1]{[#1]}

\newcommand{\RO}[1]{[#1)}

\newcommand{\Matrix}[1]{\begin{pmatrix}#1\end{pmatrix}}

\newcommand{\sMatrix}[1]{\mscriptsize{\Matrix{#1}}}

\newcommand{\set}[2][]{
    \ifthenelse{\equal{#1}{}}{
        \CB{#2}}{
        \CB{#1~|~#2}}}
\newcommand{\bset}[2][]{
    \ifthenelse{\equal{#1}{}}{
        \bCB{#2}}{
        \bCB{#1~|~#2}}}
\newcommand{\Bset}[2][]{
    \ifthenelse{\equal{#1}{}}{
        \BCB{#2}}{
        \BCB{#1~\big|~#2}}}

\DeclareMathOperator{\id}{\normalfont\msf{id}}

\renewcommand{\dim}{\operatorname{\normalfont\msf{dim}}}

\newcommand{\C}{\bb{C}}

\newcommand{\N}{\bb{N}}

\newcommand{\R}{\bb{R}}
\newcommand{\bS}{\bb{S}}

\newcommand{\Z}{\bb{Z}}

\newcommand{\cB}{\cmc{B}}

\newcommand{\sB}{\eus{B}}

\newcommand{\sF}{\eus{F}}

\newcommand{\sP}{\eus{P}}

\newcommand{\sW}{\eus{W}}

\newcommand{\G}{\Gamma}

\newcommand{\WS}{{\eus{WS}}}

\newcommand{\botimes}{\varogreaterthan}

\newcommand{\btimes}{{>}}

\newcommand{\eset}{\emptyset}

\usepackage[varg]{txfonts}

\numberwithin{equation}{section}

\begin{document}

\title[Subproduct Systems and Cartesian Systems]{Subproduct Systems and Cartesian Systems;\\new results on factorial languages\\ and their relations with other areas}



\author[Malte Gerhold]{Malte Gerhold*}
\address{MG, Institut für Mathematik und Informatik\\Universität Greifswald\\
17487 Greifswald, Germany}
\email{\href{mailto:mgerhold@uni-greifswald.de}{\tt{mgerhold@uni-greifswald.de}}}
\thanks{* The work of M.~Gerhold is partially supported by the DFG, project no. 397960675.
}
\author{Michael Skeide}
\address{MS, Dipartimento di Economia, Universit\`a\ degli Studi del Molise, Via de Sanctis, 86100 Campobasso, Italy}
\email{\href{mailto:skeide@unimol.it}{\tt{skeide@unimol.it}}}


{

\begin{abstract}
\noindent
We point out that a sequence of natural numbers is the \it{dimension sequence} of a \it{subproduct system} if and only if it is the \it{cardinality sequence} of a \it{word system} (or \emph{factorial language}). Determining such sequences is, therefore, reduced to a purely combinatorial problem in the combinatorics of words. A corresponding (and equivalent) result for graded algebras has been known in abstract algebra, but this connection with pure combinatorics has not yet been noticed by the product systems community.  We also introduce \it{Cartesian systems}, which can be seen either as a set theoretic version of subproduct systems or an abstract version of word systems. Applying this, we provide several new results on the cardinality sequences of word systems and the dimension sequences of subproduct systems.
\end{abstract}
}

\maketitle


\section{Introduction}
\label{sec:introduction}

Let $\bS$ denote the additive semigroup $\R_+=\RO{0,\infty}$ (\hl{continuous time} case) or the additive semigroup $\N_0=\CB{0,1,\ldots}$ (\hl{discrete} case). A \it{product system} of Hilbert spaces (Arveson \cite{Arv89}) is, roughly speaking, a family of Hilbert spaces $H_t$ with associative identifications
\beqn{\label{*}\tag{$*$}
H_s\otimes H_t
~=~
H_{s+t}.
}\eeqn
For a \it{subproduct system} (Shalit and Solel \cite{ShaSo09}, and Bhat and Mukherjee \cite{BhMu10}), this is weakened to
\beqn{\label{**}\tag{$**$}
H_s\otimes H_t
~\supset~
H_{s+t}.
}\eeqn
See Section~\ref{sec:prod-subpr-syst} for precise definitions.

Let $d_t:=\dim H_t$ denote the dimension of the \hl{fiber} $H_t$. Simple cardinal arithmetic shows that in a continuous time product system (henceforth, \hl{Arveson system}) the dimensions must be a constant $d:=d_t$ for $t\in(0,\infty)$. For a discrete product system we get $d_t=d_1^t$ for all $t\in\N_0$. Discrete product systems are classified by $d_1$, and $d_1$ can be any cardinal number; see Example~\ref{discPSex}. For Arveson systems, the constant $d$ can be $0$ or $1$ or any infinite cardinal number. While the finite-dimensional cases $d=0$ and $d=1$ (under measurability conditions) are understood, still after more than 20 years the case $d=\aleph_0$ seems hopeless; see Example~\ref{Fockex}. In either case, continuous time and discrete, when looking at subproduct systems we only get the very rough limitation $d_sd_t\ge d_{s+t}$, and the situation gets even more involved.

Subproduct systems generate product systems, in the sense that each subproduct system ``sits inside'' a product system in an essentially unique way; see Bhat and Mukherjee \cite{BhMu10}. In the 2009 Oberwolfach Mini-Workshop on ``Product Systems and Independence in Quantum Dynamics'' \cite{BFS09r}, Bhat suggested to try to classify at least the finite-dimensional subproduct systems and the product systems they generate. Finite-dimensional subproduct systems occurred in several ways. For instance, every CP-semi\-group on the \nbd{n\times n}matrices $M_n$ gives rise to its finite-dimensional subproduct system of \it{Arveson-Stinespring correspondences} (which, for $M_n$, are Hilbert spaces); see Shalit and Solel \cite{ShaSo09}. Moreover, every subproduct system (finite-dimensional or not) arises in this way from a normal CP-semigroup on $\sB(H)$; see again \cite{ShaSo09}. Other examples arise from homogeneous relations on polynomials in several variables; see Davidson, Ramsey and Shalit \cite{DRS11}. Also a subclass of \it{interacting Fock spaces} gives rise to finite-dimensional subproduct systems and further generalizes the notion of subproduct system; see Gerhold and Skeide \cite{GeSk13p}. Tsirelson has determined the structure of two-dimensional discrete subproduct systems \cite{Tsi09p1} and of two-dimensional continuous time subproduct systems \cite{Tsi09p2} and the product systems they generate. He exploits that subproduct systems also may be viewed as graded algebras; see also Section~\ref{sec:excluded-words}.

In these notes, we are interested in finite-dimensional discrete subproduct systems. However, instead of trying to classify them up to isomorphism, we ask which are the possible \hl{dimension sequences} $d_n=\dim H_n$. In Theorem~\ref{sec:cart-syst-word:main-theorem}, we formulate the crucial result that this problem is equivalent to the problem of finding the possible \it{cardinality sequences} of \it{word systems} over finite alphabets. (The idea of the proof is not new.\footnote{Every discrete subproduct system $(H_n)_{n\in\mathbb N_0}$ induces the structure of a \emph{standard graded algebra} on $\bigoplus_{n\in\mathbb N_0} H_n$, that is, $H_mH_n=H_{m+n}$ and $H_0=\mathbb C 1$, and it is shown by Anick \cite{Ani82} that for every standard graded algebra with dimension sequence $d_n$ there is a word system with the same sequence as cardinality sequence. Before Anick,  Stanley proved the same in the commutative case \cite{Sta78} and attributes the result to MacAulay \cite{McA27}. Our presentation here appears to be easier to follow than the original work of Anick \cite{Ani82}.})

A word system (also known as \it{factorial language}) is, roughly speaking, a set $X^\btimes$ of words over an alphabet $A$ such that all subwords of a word in $X^\btimes$ again are words in $X^\btimes$. It is well known that $X^\btimes$ is the set of all words in $A$ that do not have subwords from a certain (unique reduced) exclusion set $R$. In the latter formulation, determining cardinality sequences is a long-known and in general unsolved problem in the combinatorics of words; see, for instance, Odlyzko \cite{Odl85}. We survey these and other known results in Section~\ref{sec:excluded-words}.

Each subproduct system is isomorphic to a subproduct system in \it{standard form}. The new notion of \it{Cartesian system} is obtained, roughly, by replacing in the notion of subproduct systems Hilbert spaces with sets, tensor products with set products, and isometries with injections. In Section~\ref{CSsec}, where we study Cartesian systems and present our proof of Theorem~\ref{sec:cart-syst-word:main-theorem}. We also point out that word systems are to Cartesian systems what subproduct systems in standard form are to subproduct systems. Likewise, we show (Corollary~\ref{CSisoWScor}) that every Cartesian system is isomorphic to a Cartesian system in standard form, that is to a word system. In Section~\ref{sec:card-sequ-cart}, we present some new results. Some of these results require, as an intermediate step, to know Cartesian systems.  Therefore, the new notion of Cartesian systems promises to be a fresh contribution to combinatorics, too.  In Section~\ref{sec:count-with-word}, we exploit the fact that directed graphs (not multigraphs, but possibly with loops) give a subclass of word systems, the \it{graph systems}. These provide examples which show that some of the criteria in Section~\ref{sec:card-sequ-cart} are not necessary and others are not sufficient. The new result Theorem \ref{theo:local_to_global} asserts that for being able to realize a sequence $d_1,d_2,\ldots$ with a word system, it is sufficient ``to realize each finite part $d_1,\ldots,d_k$ $(k\in\N)$''.

To summarize: In these notes we present a proof for the---in principle known, but yet unrecognized in this context---result  that the dimension sequences of subproduct systems are precisely the cardinality sequences of word systems (Theorem~\ref{sec:cart-syst-word:main-theorem}). In the remaining sections, we make the connection with existing literature on the combinatorial problem and we present some fresh results. Most noteworthy are results based on the new and more flexible notion of Cartesian system. Relations with CP-semigroups on $M_n$ and graph \nbd{C^*}algebras are imminent, but will have to wait for future work. We, too, do not tackle the problem to classify subproduct systems to a fixed dimension sequence.


\section{Product systems and subproduct systems}
\label{sec:prod-subpr-syst}

A precise version of the identification in \eqref{*} is as follows.

\begin{definition}\label{def:product_system}
A \hl{product system} (over $\bS$) is a family $H^\otimes=\bfam{H_t}_{t\in\bS}$ of Hilbert spaces $H_t$ with $H_0=\C$ and with unitaries
\beqn{
u_{s,t}
\colon
H_s\otimes H_t
~\longrightarrow~
H_{s+t}
}\eeqn
such that the \hl{product} $x_sy_t:=u_{s,t}(x_s\otimes y_t)$ is associative and such that $u_{0,t}$ and $u_{t,0}$ are the canonical identifications.
\end{definition}

\SQ{
\begin{note}
Arveson \cite{Arv89} gave the first formal definition of product system (including also some technical conditions) of Hilbert spaces. He showed how to construct such \it{Arveson systems} from so-called normal \it{\nbd{E_0}semigroups} (semigroups of normal unital endomorphisms) over $\bS=\R_+$ on $\sB(H)$. Bhat \cite{Bha96} generalized this to normal \it{Markov semigroups} (semigroups of normal unital CP-maps) on $\sB(H)$, by dilating the Markov semigroup in a unique minimal way to an \nbd{E_0}semigroup and computing the Arveson system of the latter. Product systems of \it{correspondences} (that is, Hilbert bimodules) occur first in Bhat and Skeide \cite{BhSk00}. They constructed directly from a Markov semigroup on a unital \nbd{C^*}algebra or a von Neumann algebra $\cB$ a product system of correspondences over $\cB$, and used it to construct the minimal dilation. Muhly and Solel \cite{MuSo02} constructed from a Markov semigroup on a von Neumann algebra $\cB$ a product system over the commutant of $\cB$. This product system turned out to be the \it{commutant} (see Skeide \cite{Ske03c,Ske09,Ske08}) of the product system constructed in \cite{BhSk00}.
\end{note}
}

\begin{example}\label{Fockex}
For the \hl{continuous time case}, $\bS=\R_+$, a particular class of examples is given by families of symmetric Fock spaces $H_t:=\G(L^2(\RO{0,t},K))$ for some fixed Hilbert space $K$. The product is given by the following chain of canonical isomorphisms $H_s\otimes H_t\cong\G(L^2(\RO{t,s+t},K))\otimes H_t\cong H_{s+t}$. In principle, product systems of this type are known since Streater \cite{Str69}, Araki \cite{Ara70}, Guichardet \cite{Gui72}, or Parthasarathy and Schmidt \cite{PaSchm72}. There are by far more examples of Arveson systems than only Fock spaces; for instance, Tsirelson \cite{Tsi00p1,Tsi00p2}, Liebscher \cite{Lie09}, Powers \cite{Pow04}, Bhat and Srinivasan \cite{BhSr05}, and Izumi and Srinivasan \cite{IzSr08}. As we will mainly discuss the discrete case, we do not go into details.
\end{example}

\begin{example}\label{discPSex}
Discrete product systems $H^\otimes=\bfam{H_n}_{n\in\N_0}$ are easy to understand. If we identify $H_n$ with $H_1^{\otimes n}$ $(n\ge1)$ via the inverse of the unitary determined by $x_n\otimes\ldots\otimes x_1\mapsto x_n\ldots x_1$, it is clear that the product of $H^\otimes=\bfam{H_1^{\otimes n}}_{n\in\N_0}$ is nothing but the \it{tensor} product $(x_m\otimes\ldots\otimes x_1)(y_n\otimes\ldots\otimes y_1)=x_m\otimes\ldots\otimes x_1\otimes y_n\otimes\ldots\otimes y_1$. By mentioning that we are working in a tensor category, this is nothing but the identification map $H_1^{\otimes m}\otimes H_1^{\otimes n}\equiv H_1^{\otimes(m+n)}$. We say that a (discrete) product system $H^\otimes=\bfam{H^{\otimes n}}_{n\in\N_0}$ with the identity as product is in \hl{standard form}.
\end{example}

As far as discrete product systems are concerned, there is not more to be said than what is said in the preceding example. The situation gets more interesting for \it{subproduct systems}. A precise version of the identification in \eqref{**} is as follows.

\begin{definition}\label{def:subproduct_system}
A \hl{subproduct system} (over $\bS$) is a family $H^\botimes=\bfam{H_t}_{t\in\bS}$ of Hilbert spaces $H_t$ with $H_0=\C$ and with coisometries
\beqn{
w_{s,t}
\colon
H_s\otimes H_t
~\longrightarrow~
H_{s+t}
}\eeqn
such that the \hl{product} $x_sy_t:=w_{s,t}(x_s\otimes y_t)$ is associative and such that $w_{0,t}$ and $w_{t,0}$ are the canonical identifications.
\end{definition}

It is more common to write subproduct systems with the adjoint isometries $v_{s,t}:=w_{s,t}^*\colon H_{s+t}\to H_s\otimes H_t$, which have to fulfill the coassociativity and marginal conditions expressed in the following two diagrams.

\begin{equation}
\label{D1}
  \begin{tikzcd}[column sep=tiny]
    &H_{r+s+t}\dlar[swap]{v_{r+s,t}}\drar{v_{r,s+t}}&\\
    H_{r+s}\otimes H_t\drar[swap]{v_{r,s}\otimes\operatorname{id}_{H_t}}&&H_{r}\otimes H_{s+t}\dlar{\operatorname{id}_{H_r}\otimes v_{s,t}}\\
    &H_{r}\otimes H_s \otimes H_t
  \end{tikzcd}
\end{equation}
\begin{equation}
\label{D2}
  \begin{tikzcd}
{} &H_{t}\dlar[swap]{v_{0,t}}\drar{v_{t,0}}\dar{\operatorname{id}}&
\\
    H_{0}\otimes H_t\rar{\cong}&H_t&H_{t}\otimes H_{0}\lar[swap]{\cong}
  \end{tikzcd}  
\end{equation}
The isometries $v_{s,t}$ emphasize the idea, informally expressed in \eqref{**}, of considering $H_{s+t}$ as a subspace of $H_s\otimes H_t$.

\SQ{\lf
\begin{note}
Starting from basic examples in single operator theory, experience shows that \it{dilations} of ``irreversible'' things (like contractions, CP-maps or CP-semigroups, and so forth) to ``less irreversible'' ones (like isometries or unitaries, endomorphism semigroups or automorphism semigroups, and so forth) can be obtained by constructions that involve inductive limits. Similarly, in a number of instances, it occurred that in order to obtain a product system, one, first, has to construct a subproduct system and, then, perform a suitable inductive limit. The first construction of this type is probably Schürmann's reconstruction theorem for \hl{quantum Lévy processes} on bialgebras starting from the GNS-constructions for the marginal distributions of the process; see Schürmann \cite[Section~1.9, pp.\ 38-40]{MSchue93}. Bhat and Skeide \cite{BhSk00} construct the product system of a CP-semigroup starting from the GNS-correspondences of the individual CP-maps, while Muhly and Solel \cite{MuSo02} start from the Stinespring construction enriched by Arveson's \it{commutant lifting} \cite{Arv69}. Skeide \cite{Ske06d} constructed the \hl{product} of spatial product systems by an inductive limit over a subproduct system.

Almost at the same time, Shalit and Solel \cite{ShaSo09} (motivated by \cite{BhSk00} and \cite{MuSo02}), and Bhat and Mukherjee \cite{BhMu10} (motivated by \cite{Ske06d}) formalized this pre-product system structure, calling it \it{subproduct system} (\cite{ShaSo09}, immediately for correspondences) or \it{inclusion system} (\cite{BhMu10}, for Hilbert spaces). The purposes in \cite{ShaSo09} and \cite{BhMu10} are different, the results numerous. Shalit and Solel \cite{ShaSo09} also consider subproduct systems over more general monoids $\bS$, generalizing the corresponding definition of product systems by Fowler \cite{Fow02}.
\end{note}
}

\noindent
It is a basic problem to embed a subproduct system into a product system. Bhat and Mukherjee \cite[Theorem 5]{BhMu10}, inspired by the inductive limit constructions in \cite{MSchue93,BhSk00,MuSo02,Ske06d}, formalized this idea:

\begin{theorem}\label{SPS->PSthm}
Every continuous time subproduct system is isomorphic to a subproduct subsystem of a product system. The product subsystem \hl{generated} by the subproduct system is determined up to a unique isomorphism intertwining the contained subproduct system.
\end{theorem}

Here, an \hl{isomorphism} of subproduct systems is a family of unitaries, intertwining the products. (If one of the subproduct systems is a product system, then so is the other and the isomorphism is actually an isomorphism of product systems.) For the definition of a \hl{subproduct subsystem} $H'^\botimes$ of a subproduct system $H^\botimes$, there is the subtlety that we have to distinguish between invariance of the family of Hilbert subspaces $H'_t\subset H_t$ for the coisometric product maps $w_{s,t}$ and invariance for the isometric coproduct maps $v_{s,t}$. (\it{A priori}, the latter is stronger a condition than the former. See \cite[Definition 5.1]{ShaSo09}.  And \cite{BhMu10} even use the term without definition.) Fortunately, since for product systems the $v_{s,t}$ and the $w_{s,t}$ are unitaries, for
a family of Hilbert subspaces $H'_t\subset H_t$ for 
a product system $H^\otimes$ we have the following obvious equivalence.

\begin{proposition}
~~~$v_{s,t}H'_{s+t}\subset H'_s\otimes H'_t$~~~ if and only if ~~~$w_{s,t}(H'_s\otimes H'_t)\supset H'_{s+t}$.
\end{proposition}

\SQ{
\begin{note}
The possibility to prove Theorem~\ref{SPS->PSthm} depends on the order structure of the monoid $\bS$. (It does not matter, instead, if we speak about Hilbert spaces or correspondences.) Roughly speaking, the interval partitions for a properly defined partial order have to form a directed set. (See Shalit and Skeide \cite{ShaSk10p} for details.) Every Markov (or just contractive CP-)semigroup comes along with the subproduct system of GNS or Arveson-Stinespring correspondences. But to find a dilation, it is crucial to embed one of these subproduct systems into a product system. The fact that the interval partitions for the monoids $\N_0^k$ and $\R_+^k$ $(k\ge2)$ are no longer a directed set, motivated \cite{ShaSo09} to define first subproduct systems and to analyze their structure. (See also Shalit and Skeide \cite{ShaSk10p} for details.)
\end{note}
}

\noindent
For discrete subproduct systems, the situation is even simpler: The interval partitions of the segment $\SB{0,n}\cap\N_0$ are a finite lattice and, therefore, have a unique maximum. Together with Example~\ref{discPSex}, which describes the simple structure of discrete product systems, we recover the result \cite[Lemma 6.1]{ShaSo09}:

\begin{theorem}\label{SPS->standthm}
Every discrete subproduct system $H^\botimes=\bfam{H_n}_{n\in\N_0}$ is isomorphic to a subproduct subsystem of a product system $H^\otimes=\bfam{H^{\otimes n}}_{n\in\N_0}$ in standard form (Example \ref{discPSex}).
\end{theorem}

Obviously, $H$ can be chosen to be $H_1$.

\begin{definition}\label{def:standard_form}
We say a subproduct subsystem of a product system in standard form is in \hl{standard form}.
\end{definition}


\section{Cartesian systems and word systems}\label{CSsec}

\it{Cartesian systems} are the analogue of subproduct systems, where we replace Hilbert spaces $H_n$ with sets $X_n$ and tensor products with set products (Definition~\ref{CSdefi}). \it{Word systems} are Cartesian systems in standard form (Theorem~\ref{WS=CSSFthm}). We also will see (Corollary~\ref{CSisoWScor}), in analogy with Theorem~\ref{SPS->standthm}, that every Cartesian system is isomorphic to a word system. But in the first place, we discuss the crucial Theorem~\ref{sec:cart-syst-word:main-theorem}: The dimension problem for finite-dimensional subproduct systems is equivalent to the cardinality problem for word systems over finite alphabets.

\begin{definition}\label{CSdefi}
A \hl{Cartesian system} (over $\bS$) is a family $X^\btimes=\bfam{X_t}_{t\in\bS}$ of sets $X_t$ with $X_0=\gklammer{\Lambda}$, a one point set, and with injections
\beqn{
i_{s,t}
\colon
X_{s+t}
~\longrightarrow~
X_s\times X_t
}\eeqn
such that the following analogues of Diagrams \eqref{D1} and \eqref{D2} commute.
    \begin{equation}
      \begin{tikzcd}[column sep=tiny]
        &X_{r+s+t}\dlar[swap]{i_{r+s,t}}\drar{i_{r,s+t}}&\\
        X_{r+s}\times X_t\drar[swap]{i_{r,s}\times\operatorname{id}_{X_t}}&&X_{r}\times X_{s+t}\dlar{\operatorname{id}_{X_r}\times i_{s,t}}\\
        &X_{r}\times X_s \times X_t
      \end{tikzcd}
    \end{equation}
    \begin{equation}
      \begin{tikzcd}     
{}&X_{t}\dlar[swap]{i_{0,t}}\drar{i_{t,0}}\dar{\operatorname{id}}&
\\
        X_{0}\times X_t\rar{\cong}&X_t&X_{t}\times X_{0}\lar[swap]{\cong}
      \end{tikzcd}
    \end{equation}
A family $Y^\btimes=\bfam{Y_t}_{t\in\bS}$ of subsets $Y_t\subset X_t$ is a \hl{Cartesian subsystem} of  $X^\btimes$ if $i_{s,t}(Y_{s+t})\subset Y_s\times Y_t$, where, as customary, we identify $Y_s\times Y_t\subset X_s\times X_t$.
\end{definition}

In the following we will consider Cartesian systems over $\N_0$ only. 

\it{Word systems} are for Cartesian systems what subproduct systems in
standard form are for subproduct systems. Let us fix a set $A$, the \hl{alphabet}. Put $A^0:=\CB{\Lambda}$ where $\Lambda:=()$ is the \hl{empty tuple}. Denote by $A^*:=\bigcup_{n\in\N_0} A^n$ the set of all finite \hl{words} with \hl{letters} $a$ in $A$. Denote the \hl{length} of a word $w\in A^n$ by $\betrag w:=n$. Defining a product by \it{concatenation}
\begin{equation*}
  (a_1,\ldots, a_n)(b_1,\ldots, b_m):=(a_1,\ldots,a_n,b_1,\ldots,b_m),
\end{equation*}
we turn $A^*$ into a monoid with $\Lambda$ as neutral element.

\begin{remark}
We apologize to all people working in \it{word theory} for not writing a word as $w=a_1\ldots a_n$. Our choice underlines the analogy with subproduct systems. And many of the following formulations run more smoothly when there is a product of words, but no product of letters.
\end{remark}

\begin{proposition}\label{prop:word_factorization}{~}
  \begin{enumerate}
\item\label{P1}
Every word $w$ in $A^{\ell_1+\ldots+\ell_k}$ factors uniquely as $w=w_1\ldots w_k$ with $w_i\in A^{\ell_i}$.
\item\label{P2}
Suppose $X_{i}\subset A^{\ell_i}$. Then $w\in
X_{1}\ldots X_{k}\subset A^{\ell_1+\ldots+\ell_k}$ if and only if $w_{i}\in X_{i}$ for all $i=1,\ldots,k$.
\end{enumerate}
\end{proposition}

\begin{proof}
This proposition is a simple consequence of the fact that $A^{\ell_1}\times\ldots\times A^{\ell_k}$ may be identified with $A^n$ by sending $(w_1,\ldots,w_k)$ to $w_1\ldots w_k$, and of the fact that an element $s$ in a product $S_1\times \ldots\times S_k$ is a unique tuple
$(s_1,\ldots, s_k)$.
\end{proof}

We say a word $y$ is a \hl{subword} of $w$ if there are words $x,z\in A^*$ with $w=xyz$. One may check that the relation defined by $y$ being a subword of $x$, is a partial order.

\begin{definition}\label{def:word_system}
Let $A$ be an alphabet. A family $X^{\btimes}=\bfam{X_n}_{n\in\N_0}$ of subsets $X_n\subset A^n$ with $X_0=A^0$ is called a \hl{word system} over $A$ if it is \hl{closed under building subwords}. Writing $w\in X^\btimes$ if $w\in X_n$ for some $n\in\N_0$, this means that
  \begin{equation*}
y\in X^{\btimes}
\text{~~~whenever~~~}
xyz\in X^{\btimes}.
  \end{equation*}
 for some $x,z\in A^*$.
\end{definition}

\begin{proposition}\label{recurprop}
$X^{\btimes}$ is a word system if $X_{n+1}\subset AX_n\cap X_nA$ for all $n\in\N_0$.
\end{proposition}

\begin{proof}
By repeated application of the inclusion, we obtain $X_{m+n+k}\subset A^mX_nA^k$. So, by Proposition~\ref{prop:word_factorization}\eqref{P2}, if $xyz\in X_{m+n+k}\subset A^mX_nA^k$ with $x\in A^m, y\in A^n, z\in A^k$, then $y\in X_n$.
\end{proof}

Obviously, also the converse is true.

\begin{example}\label{FWSex}
Choose an alphabet $A$. By Proposition~\ref{prop:word_factorization}\eqref{P1}, the restriction of the product to $A^n\times A^m$ is an invertible map onto $A^{n+m}$. Define $i_{n,m}$ to be the inverse of this map. Then $A^\times:=\bfam{A^n}_{n\in\N_0}$ with the maps $i_{n,m}$ is a Cartesian system, the \hl{full word system over $A$}.
\end{example}

Full word systems play the role of product systems in standard form. We now show that word systems play the role of subproduct systems in standard form. Recall that the structure maps $i_{n,m}$ of $A^\times$ are the inverses of the restricted product maps. In particular, they are invertible so that $X_n\subset A_n$ form a Cartesian subsystem of $A^\times$ if and only if $X_{n+m}\subset X_nX_m$.

\begin{theorem}\label{WS=CSSFthm}
Let $A$ be an alphabet. For a family $X^{\btimes}=\bfam{X_n}_{n\in\N_0}$ of subsets $X_n\subset A^n$ the following conditions are equivalent.
  \begin{enumerate}
  \item\label{it_ws}
$X^{\btimes}$ is a word system over $A$.

  \item\label{it_cs}
$X^{\btimes}$ is a Cartesian subsystem of the full word system $A^\times$ over $A$.
  \end{enumerate}
\end{theorem}

\begin{proof}
If $w\in X^\btimes$ and $w=xy$, then $x$ and $y$ are subwords. So, \ref{it_ws}$\Rightarrow$\ref{it_cs} is immediate.

 Conversely, suppose $X_{m+n}\subset X_mX_n$ for all $m,n\in\N_0$. This means, in particular, that $X_{n+1}\subset X_1X_n\cap X_nX_1\subset AX_n\cap X_nA$. By Proposition~\ref{recurprop}, $X^\btimes$ is a word system, that is, \ref{it_cs}$\Rightarrow$\ref{it_ws}.
\end{proof}

As mentioned in the introduction, $X^\btimes\subset A^*$ being a word system is equivalent to a number of other properties well-known in the combinatorics of words. We comment on these in Section~\ref{sec:excluded-words}. We conclude the present section by examining the relationship between Cartesian systems, word systems, and the subject of our main interest: subproduct systems.

\begin{example}\label{ws-cs}
  Let $X^{\btimes}=\bfam{X_n}_{n\in\N_0}$ be a Cartesian system. Denote by $H_n$ the canonical Hilbert space with orthonormal basis $X_n$. Then, clearly, the embeddings $i_{n,m}$ of $X_{m+n}$ into $X_m\times X_n$ extend as isometries $v_{m,n}\colon H_{m+n}\rightarrow
  H_m\otimes H_n$ and the $v_{n,m}$ define a subproduct system structure, the subproduct system \hl{associated} with $X^\btimes$. Moreover, if $X^\btimes$ is a word system over $A$, so that $X_n\subset A^n$ and $H_n\subset H_1^{\otimes n}$, this subproduct system is in standard form.
\end{example}

Obviously, $\dim H_n=\#X_n$. We see, for every Cartesian (word) system there is a subproduct system (in standard form) such that the dimension sequence of the latter coincides with the cardinality sequence of the former. Before we show the converse statement in Proposition~\ref{sec:cart-syst-word:main-proposition}, let us mention that not all subproduct systems are isomorphic to one that is associated with a word system.

\begin{observation}\label{sec:cart-syst-word:obs}
If at least one $X_n$ in a word system contains a word with at least two
  different letters, then the associated subproduct system in standard
  form is not commutative. But there are commutative subproduct
  systems. See, for instance, the \it{symmetric subproduct system}
  introduced by Shalit and Solel \cite{ShaSo09}, which is
  obtained by considering the symmetric tensor power $H^{\otimes_s n}$
  as subspace of $H^{\otimes n}$.
\end{observation}

We now prepare for Proposition~\ref{sec:cart-syst-word:main-proposition}.

\begin{definition}\label{def:lexicographical_order}
  Let $A$ be a partially ordered alphabet. Then the
  \hl{lexicographical order} $\leq_{lex}$ on $A^n$ is given by
  $(a_1,\ldots, a_n)\leq_{lex}(b_1,\ldots, b_n)$ if $a_k=b_k$ for all
  $k$ or $a_k<b_k$ where $k$ is the smallest index $i$ with $a_i\neq b_i$.
\end{definition}

It is easy to show that the lexicographical order is a total order on $A^n$ whenever $\leq$ is a total order on $A$. Without the obvious proof, we state:

\begin{lemma}
  Let $y,y'\in A^n$. Then
  \begin{equation}
    y\leq_{lex}y'
~~~\Longleftrightarrow~~~
   \exists x,z\in A^*\colon xyz\leq_{lex}xy'z
~~~\Longleftrightarrow~~~
    \forall x,z\in A^*\colon xyz\leq_{lex}xy'z.
  \end{equation}
\end{lemma}

\begin{proposition}[Cf.\ {\cite[Lemma 1.1]{Ani82}}]\label{sec:cart-syst-word:main-proposition}
  Let $H^{\botimes}=(H_n)_{n\in\N_0}$ be a finite dimensional subproduct system. Then there exists a word system $X^{\btimes}=(X_n)_{n\in\N_0}$ with
  \begin{equation*}
    \# X_n=\dim H_n \qquad\forall n\in\N_0.
  \end{equation*}
\end{proposition}

\begin{proof}
  Let $(e_1,\ldots,e_d)$ be a basis of $H_1$. Set
  $A=\gklammer{1,\ldots,d}$. 
  For a word $w=(a_1,\ldots ,a_n)$ in $A^n$ we
  define the element $e_w\in H_n, e_w:=e_{a_1}\ldots e_{a_n}$. The
  multiplication of $H^\botimes$ is coisometric, hence,
  surjective. Therefore, $H_n=\Span\gklammer{e_w\mid w\in
    A^n}$.  Set
  \begin{equation*}
    X_n:=\gklammer{w\in A^n\mid e_w\notin\Span\gklammer{e_v\mid v<_{lex}w}}
  \end{equation*}
  
  Since $\gklammer{e_w\mid w\in X_n}$ is linearly independent and still
  spans $H_n$, it is a basis of $H_n$. We, thus, have $\# X_n=\dim
  H_n$ for all $n\in\N_0$.

We claim $X^{\btimes}=(X_n)_{n\in\N_0}$ is a word system over $A$. For a word
  $w\in A^n$ choose a subword $y\in A^k$ ($k\le n)$, so that there are $x,z\in
  A^*$ with $w=xyz$. We are done if we
   show $y\notin X_k\Rightarrow xyz\notin X_n$. Suppose $y\notin X_k$, that is,
  \begin{equation*}
    e_y=\sum_{y'<y}\alpha_{y'} e_{y'}.
  \end{equation*}
Then,
  \begin{equation*}
    e_w=e_xe_ye_z =\sum_{y'<y}\alpha_{y'} e_xe_{y'}e_z
    =\sum_{y'<y}\alpha_{y'} e_{xy'z}.
  \end{equation*}
  Since, by the lemma, $y'<_{lex}y$ implies $xy'z<_{lex}xyz$, we obtain $e_w\notin X_n$. 
\end{proof}

\begin{remark}
It might appear appealing, again to use Proposition~\ref{recurprop}. In this case, however, it would rather make the proof more complicated.
\end{remark}

\begin{example}
  For the symmetric subproduct system $H^{\otimes_s}$ we obtain 
  \begin{equation*}
    X_n=
    \bCB{\,(i_1,\ldots, i_n)\mid i_1\leq i_2\leq\ldots\leq i_n}.
  \end{equation*}
  As seen in Observation~\ref{sec:cart-syst-word:obs}, for $\dim H\ge2$, the subproduct system associated with $X^{\btimes}$ in the sense of
  Example~\ref{ws-cs} is not isomorphic to the
  original one.
\end{example}

Our main theorem is now a simple corollary of Example~\ref{ws-cs} and Proposition~\ref{sec:cart-syst-word:main-proposition}.

\begin{theorem}\label{sec:cart-syst-word:main-theorem}
For a seqence $\bfam{d_n}_{n\in\N_0}$ of numbers $d_n\in\N_0$ the following conditions are equivalent:
\begin{enumerate}
\item
$\bfam{d_n}_{n\in\N_0}$ is the dimension sequence of a subproduct system.

\item
$\bfam{d_n}_{n\in\N_0}$ is the cardinality sequence of a word system.
\end{enumerate}
\end{theorem}

And finally:

\begin{corollary}\label{CSisoWScor}
Every Cartesian system is isomorphic to a word system.
\end{corollary}

\begin{proof}
Let $X^\btimes$ be a Cartesian system and construct the subproduct system associated with $X^\btimes$, $H^\botimes$, as in Example~\ref{ws-cs}. Then apply the proof of Proposition~\ref{sec:cart-syst-word:main-proposition} to $H^\botimes$ and check that output is isomorphic to $X^\btimes$.
\end{proof}

Of course, this also can be proved directly without reference to subproduct systems.


\section{The cardinality sequence of a word system: Known results}
\label{sec:excluded-words}

Word systems are also known as \it{factorial languages} and subwords are known as \it{factors}; see, for instance, Crochemore, Mignosi, and Restivo \cite{CMR98}. Cardinality sequences also appear under names like (\it{combinatorial}) \it{complexity} (\it{sequence}, \it{function}).

Correspondingly, there is a long list of known results. Additionally, there are equivalent descriptions, still multiplying the number of applicable results. Unfortunately, the publications dealing with this structure (under different names or in equivalent definitions) frequently seem to not interact. (We hope it may be forgiven that we add a further name, word systems, that is inspired from the analogy with subproduct systems.) This feature does not make it particularly easy to get an idea about the real status of the theory. In this section we intend to give an overview over such relations. It should be clear that this cannot be exhaustive. But we hope we can at least provide a small guide pointing into interesting directions, and we cite sources where the interested reader can find more information.

\subsection*{Reduced sets of excluded words}
A word system can be described by indicating which words do not occur as subwords. The following results are well known (see for example \cite{CMR98}), but we prefer to give independent proofs, firstly, to illustrate how arguments work and, secondly, to be self-contained in the following section. They also promise to be relevant in analyzing the structure of associated graph \nbd{C^*}algebras.

\begin{observation}\label{ob1}
  Let $E\subset A^*$ be any set of words. Then the sets
  \begin{equation*}
    X_n(E):=\gklammer{w\in
      A^n\mid \text{$w$ has no subword from $E$}}
\end{equation*}
form a word system, denoted $X^\btimes(E)$.  Indeed, if $w$ does not contain a subword from $E$
and $y$ is a subword of $w$, then, by transitivity, also $y$ cannot contain a subword from $E$.
\end{observation}

Reflexivity means that every word is a subword of itself. Therefore, $E$ and $X^\btimes(E)$ are disjoint. 

\begin{observation}\label{ob2}
  Every word system $X^{\btimes}$ can be obtained as
  $X^\btimes=X^\btimes(E)$.  Indeed take $E:=A^*\setminus
  X^{\btimes}$, the set of all words in $A^*$ that do not belong to
  $X^\btimes$. A word belongs to the word system $X^{\btimes}$ if and
  only if all its subwords belong to $X^\btimes$. Equivalently $x\in
  X^\btimes$ if and only if none of its subwords is in $A^*\setminus
  X^\btimes$, that is, $X^\btimes=X^\btimes(A^*\setminus
  X^{\btimes})$.
\end{observation}

$E=A^*\setminus X^{\btimes}$ is, clearly, the maximal choice. We now show that there is a unique minimal choice.

\begin{definition}\label{def:reduced}
  A subset $E\subset A^*\setminus\gklammer{\Lambda}$ is called \hl{reduced} if no word of $E$ is a proper subword of another word of $E$.
\end{definition}

Reduced sets are also known as  \it{antifactorial languages}.

Note that $E=A^*\setminus X^{\btimes}$ is reduced if and only if it is empty, that is, if $X^\btimes$ is the full word system over $A$. (Indeed, suppose that $E$ is reduced. If $A=\eset$, so that $A^*=\CB{\Lambda}$, then a reduced subset $E$ of $A$, by definition, is empty. And if $A$ is nonempty, then every word $x$ is a proper subword of another word $y$. If $y\in X^\btimes$, then $x\in X^\btimes$, because $X^\btimes$ is a word system. If $y\notin X^\btimes$, that is, if $y\in E$, then $x\notin E$, that is, $x\in X^\btimes$, because $E$ is reduced. So, every word $x$ is in $X^\btimes$, that is, $X^\btimes=A^*$. The other direction is obvious.)

\begin{proposition}\label{sec:excluded-words-1}
  Let $E$ be reduced and $X^{\btimes}(E)=X^{\btimes}(E')$. Then $E\subset E'$. 
\end{proposition}

\begin{proof}
We conclude indirectly. Suppose $w\in E\setminus E'$. Since $w\in E$, $w$ does not belong to $X^\btimes(E)=X^\btimes(E')$. Therefore, $w$ contains a subword $y\in E'$. Since $w\notin E'$, $y$ is a proper subword of $w$. Since $E$ is reduced, $y$ and all subwords of $y$ are not in $E$. Therefore, $y\in X^\btimes(E)$. But, $y\in E'$, so $y\notin X^\btimes(E')=X^\btimes(E)$. Contradiction!
\end{proof}

This proposition shows that if there is a reduced set $R$ such that $X^\btimes(R)=X^\btimes$, then
$  R=\bigcap_{X^{\btimes}(E)=X^{\btimes}} E$.
In particular, $R$ is unique. The following theorem settles existence by giving an explicit formula for $R$. The unique reduced set $R$ generating $X^\btimes$ as $X^\btimes(R)$ is also called the \it{antidictionary} of $X^\btimes$.

\begin{theorem}\label{sec:excluded-words-2}
  For every word system $X^{\btimes}$ over $A$,
  \begin{align*}
    R&:=\bigcup_{n\geq1} R_n,
    &
    R_n&:=(X_{n-1} A\cap A X_{n-1})\setminus X_{n}
  \end{align*}
  is the unique reduced set of words such that $X^\btimes=X^\btimes(R)$.
\end{theorem}

\begin{proof}
A word $w=(a_1,\ldots, a_n)$ is in $X_{n-1} A\cap A X_{n-1}$ if and only if the two subwords $w_{\wh{n}}=(a_1,\ldots, a_{n-1})$ and $w_{\wh{1}}=(a_2,\ldots, a_n)$ are in $X_{n-1}$. Now, each proper subword $y$ of $w$ is a subword of $w_{\wh{n}}$ or a subword of $w_{\wh{1}}$. Since $X^\btimes$ is a word system, $y\in X^\btimes$. In other words, $w=(a_1,\ldots, a_n)$ is in $X_{n-1} A\cap A X_{n-1}$ if and only if each of its proper subwords is in $X^\btimes$.

In order to illustrate some different techniques, we continue in two versions.

\lf
\ul{Version 1:} Since all proper subwords of $w\in R_n$ are in $X^\btimes$, these subwords are not in $R$. Therefore, $R$ is reduced.

  To show $X^{\btimes}\subset X^{\btimes}(R)$ take any $w\notin X^{\btimes}(R)$. Then $w$ has a subword $r\in R$. Since $R$ and $X^\btimes$ are disjoint, $r$ is not in $X^\btimes$. Hence, the word $w$ containing $r$ is not in the word system $X^\btimes$. 

  For the other inclusion, we show $X_n(R)\subset X_n$ by induction on
  $n$. Since a reduced set may not
  contain the empty word, $X_0=\gklammer{\Lambda}=X_0(R)$. Now let $n\ge1$ and suppose $X_{n-1}(R)\subset X_{n-1}$. Let
  $w=(a_1,\ldots, a_n)\in X_n(R)$. As $X^{\btimes}(R)$ is a
  word system, the two subwords $w_{\wh{n}}$ and $w_{\wh{1}}$ of $w$ belong to $X_{n-1}(R)$. By assumption, $X_{n-1}(R)$ is a subset of $X_{n-1}$. In other words, $w\in(X_{n-1}A)\cap (AX_{n-1})$.  Since $R$
  and $X^\btimes(R)$ are disjoint, $w$ is not an element of
  $R_n$. Since $R_n=((X_{n-1}A)\cap (AX_{n-1}))\setminus X_n$,
  this implies $w\in X_n$, so $X_n(R)\subset X_n$ for all $n$. In conclusion, $X^\btimes(R)\subset X^\btimes$.

\lf
\ul{Version 2:} For any $E\subset A^*$ we may obtain the unique reduced $E^{red}$ such that $X^\btimes(E^{red})=X^\btimes(E)$ by replacing $E_n:=E\cap A^n$ with
\beqn{
E^{red}_n:=\bCB{w\in E_n\mid\text{$w$ has no subword from $E_k$, $k=1,\ldots,n-1$}}.
}\eeqn
(We omit the proof.) So, for our word system $X^\btimes$, appealing to Observation~\ref{ob2}, put $E:=A^*\setminus X^\btimes$. We find
\baln{
E^{red}_n
&
=\bCB{w\in A^n\setminus X_n\mid\text{$w$ has no subword from $A^k\setminus X_k$, $k=1,\ldots,n-1$}}.
\\
&
=\bCB{w\in A^n\setminus X_n\mid\text{all proper subwords of $w$ are in $X^\btimes$}}.
\\
&
=\bCB{w\in A^n\setminus X_n\mid w\in X_{n-1}A\cap AX_{n-1}}.
}\ealn
So, $E^{red}_n=R_n$.
\end{proof}

The second proof also illustrates the feature of exclusion sets with only one word as \it{atoms}, and further exploitation of the problem's \it{inductive structure} will be demonstrated in Theorem~\ref{theo:local_to_global}. For simplicity assume $X_1=A$. To understand the reduced set $R$ of a given word system $X^\btimes$ over $X_1$, for $R_2$ simply take all words of length $2$ that do not occur in $X^\btimes$. Then to get $R_3$ from $X^\btimes(R_2)$ take all words of length $3$ that do not occur in $X^\btimes$. Then proceed with $X^\btimes(R_2\cup R_3)$ and words of length $4$ to get $R_4$, and so forth. In general, we have
\beqn{
X^\btimes(E\cup E')
~=~
X^\btimes(E)\cap X^\btimes(E').
}\eeqn
So, not only do we get $X^\btimes(R_2\cup\ldots\cup R_n)=X^\btimes(R_2)\cap\ldots\cap X^\btimes(R_n)$, but
\beqn{
X^\btimes
~=~
\bigcap_{r\in R}X^\btimes(\CB{r}).
}\eeqn
So, being the smallest building blocks (the maximal proper word subsystems) it is important to understand first the the cases $R=\CB{r}$ ($X^\btimes(\CB{r})$). Also the case where $R=R_2$ is important; in Section~\ref{sec:count-with-word} it will lead to word systems of graphs.

\subsection*{A sufficient condition for boundedness} Balogh and Bollob{\'a}s \cite[Theorems 6 and 8]{BaBo05} prove that if $\#X_k\leq k$ for some $k\in\mathbb{N}$, then $(\#X_n)_{n\in\mathbb{N}}$ is bounded; more precisely, if $\#X_k\leq k$, then
\begin{align}
  \label{eq:BaBo}
  \#X_\ell\leq \left\lceil\frac{\#X_k+1}{2}\right\rceil\left\lfloor\frac{\#X_k+1}{2}\right\rfloor 
\end{align}
for all $\ell\geq k + \#X_k$, where $\lceil n\rceil,\lfloor n\rfloor$ denote the smallest integer $\geq n$ and the largest integer $\leq n$, respectively. This result becomes very powerful when one deals with the monoids $\mathbb{Q}_+$ or $\mathbb{R}_+$ instead of $\mathbb{N}_+$, see Note~\ref{note:dimension-rational-time}.

\subsection*{Generating functions} Guibas and Odlyzko find
the generating function $ \sum_{n=0}^\infty \# X_n(R)z^{-n}$
 for a word system with a finite reduced set $R$ of excluded words as the solution of a system of linear equations only depending on the so-called \it{correlation} of the words in $R$ \cite[Theorem 1.1]{GuOd81}. As a special case, they give an explicit formula for the generating function in the case $R=\CB{r}$, which depends only on the \it{autocorrelation} of the only one excluded word $r$. In \cite[Section 7]{GuOd81}, they decide which word $r$ gives the ``biggest'' word system:  $\# X_n(\gklammer{r})\geq \# X_n(\gklammer{s})$ if and only if the autocorrelation of $r$ is less or equal to the autocorrelation of $s$. There is a nice survey in Odlyzko \cite{Odl85}. Some more methods to determine the generating function can be found in Goulden and Jackson \cite{GoJa79}.

\subsection*{Growth rates}
One may analyze the asymptotic behaviour of the cardinality sequence. It is clear, that the sequence may break down (simply replace $X_n$ with $\emptyset$ for all $n\geq N$), or that $d_n$ is limited by $d^n$ for the full word system $A^*$ with $\#A=d$. But there are more interesting results. For instance, Shur shows in  \cite{Shu06} that for all $s\in\R_+$ there are word systems with \it{asymptotic growth rate} $n^s$. In \cite{Shu09}, he shows that there are word systems with asymptotic growth rate larger than every polynomial and smaller than every exponential function.

\subsection*{Subword complexities}
It seems that it is easier to get estimates when restricting to the subclass of word systems  $X^{w\btimes}$ consisting of all (finite) subwords of a single (usually) infinite word $w$, in which case the cardinality sequence is referred to as \it{subword complexity}. (These word systems are particularly relevant for computer science.) As early as 1938, Morse and Hedlund \cite{MoHe38} provided a necessary condition for that a sequence occurs as subword complexity: Either $\#X^w_{n+1}>\#X^w_n$ for all $n\in\mathbb{N}$ or $\#X^w_n$ is eventually constant.  Ferenczi \cite[Section 3]{Fer99} provides several results on  subword complexities. For instance,
if $\#X_n^w\leq an$ for all $n\in\N$, then there is $C$ such that $\#X_{n+1}^w-\#X_n^w\leq Ca^3$ for all $n$.


\section{The cardinality sequence of a word system: Some new results}
\label{sec:card-sequ-cart}

In this section we present some results which, we believe, may be new. The results are formulated for cardinality sequences of word systems. Of course, from Corollary~\ref{CSisoWScor} it follows that all these results remain true for Cartesian systems, and from Theorem~\ref{sec:cart-syst-word:main-theorem} it follows that all these results remain true for subproduct systems replacing `cardinality sequence' with `dimension sequence'. It also should be noted that some results and their consequences are are much easier to prove for Cartesian systems than for word systems. (For Cartesian systems, Theorem~\ref{(dn)thm} is a triviality while the proof of Theorem~\ref{(n+k)thm} runs very smoothly. For both we cannot even imagine how to write down a proof using only word systems.)

\subsection*{Local to global}

Let $X^\btimes$ be a word system over $A$ and let $R$ be the unique reduced set of excluded words such that $X^\btimes=X^\btimes(R)$. It is noteworthy that in order to determine $X_i$ and, therefore, $d_i=\#X_i$ for $i=1,\ldots,k$, we only need to know $R_i$ for $i=1,\ldots,k$. In order to `realize' the partial sequence $d_1,\ldots,d_k$, it does not matter what the word system $X^\btimes$ does for $i>k$, nor, equivalently, what the $R_i$ are for $i>k$. We may cut down $X^\btimes$ by assuming $X_i=\eset$ for $i>k$; this is an easy choice but, possibly, not the most clever, because it makes the corresponding $R_i$ rather big. We also may cut down $R$ by assuming $R_i=\eset$ for $i>k$; this gives the biggest word system with the partial sequence $X_i$ for $i=1,\ldots,k$ with the corresponding $R_i$ for $i=1,\ldots,k$. This choice has the advantage that now the resulting truncated set of excluded words is finite, so, all results for generating functions for finite sets of excluded words (for instance, those in \cite{GuOd81}) are applicable for checking if the partial sequence $d_1,\ldots,d_k$ can be realized for suitable choices of $R_1,\ldots,R_k$.

If, for a sequence, $d_1,d_2,\ldots$, we can realize $d_1,\ldots,d_k$ for each $k$ by choosing $R_1,\ldots,R_k$ in such a way that for $k+1$ we just add $R_{k+1}$ to $R_1,\ldots,R_k,$ then, of course, the whole sequence of $R_k$ obtained in that inductive way determines a word system $X^\btimes$ with $\#X_k=d_k$ for all $k$. But what, if we can realize each finite subsequence $d_1\ldots,d_k$ but without being able to fix the $R_i$ for $i\leq k$ in the next step? The following theorem shows that this \it{local realizability} of the sequence $d_1,d_2,\ldots$ is sufficient.

\begin{theorem}\label{theo:local_to_global}
  Let $\bfam{d_n}_{n\geq1}$ be a sequence of nonnegative
  integers. Suppose for every $k\in\mathbb{N}$ there exists a word
  system $Y^{(k)\btimes}$ with $\# Y^{(k)}_i=d_i$ for $i=1,\dots, k$. Then there
  exists a word system $X^\btimes$ with $\# X_i=d_i$ for all $i\geq
  1$.
\end{theorem}

\begin{proof}
 Every word system $X^\btimes$ may be considered as a word system over $X_1$. And if two word systems, say $X^\btimes$ and $X'^\btimes$, fulfill $\#X_1=\#X'_1$, then any bijection from $X'_1$ to $X_1$ corresponds to an isomorphism from $X'^\btimes$ to a word system over $X_1$, $X''^\btimes$, isomorphic to $X'^\btimes$. So, we may assume that the $Y^{(k)\btimes}$ realizing $d_1,\ldots,d_k$ are over a fixed finite alphabet $A$.

Let us consider word systems as elements of the product
  \begin{equation*}
    \sW(A):=\bigtimes_{n\in \N_0}\sP(A^n)
  \end{equation*}
  of the power sets $\sP(A^n)$ of $A^n$. By $\WS(A)\subset\sW(A)$ we denote the set of all word
  systems over $A$. Let $\bfam{Y^{(k)\btimes}}_{k\in\N}$ be a
  sequence in $\WS(A)$ fulfilling $\# Y^{(k)}_{i}=d_i$ for
  $i=1,\dots, k$. We will show:
  \begin{enumerate}
  \item
There is a subsequence $\bfam{Y^{(k_n)\btimes}}_{n\in\N}$ of $\bfam{Y^{(k)\btimes}}_{k\in\N}$ such that for each $i\in\N$ the sequence $\bfam{Y^{(k_n)\btimes}_i}_{n\in\N}$ is eventually constant, say, $Y^{(k_n)\btimes}_i=:X_i$ for sufficiently big $n$.

  \item
The $X_i$ form a word system $X^\btimes$ with $\#X_i=d_i$.
  \end{enumerate}
Such a sequence can be constructed explicitly by hand. But one has to introduce \it{ad hoc} total orders on $\sP(A^n)$, and writing it down requires lots of more indices. We prefer to introduce a topology on $\sW(A)$ that allows to apply Tychonov's theorem.

We equip $\sP(A^n)$ with the discrete topology and $\sW(A)$ with the product topology. So, convergence in $\sP(A^n)$ means eventually constant, and convergence in $\sW(A)$ means eventually constant entry-wise. Since $A$ is assumed finite, $\sP(A^n)$ is finite, hence, compact. By Tychonov's theorem, $\sW(A)$ is compact. Since $\sW(A)$ is first countable, it is even sequentially compact. This proves (1) and, of course, it proves that the limit $X^\btimes$ of the subsequence of $\bfam{Y^{(k)\btimes}}_{k\in\N}$ fufills $\#X_i=d_i$ for all $i$.

To show that $X^\btimes$ is a word system, we show that $\WS(A)$ is closed in $\sW(A)$. Suppose $Z\in\sW(A)$ is not a word system. That is, there exists a word $w\in Z_k$ with a subword $y$ of $w$ with $y\in A^m\setminus Z_m$. Then the set $U:=\gklammer{Z_0}\times\gklammer{Z_1}\times\dots\times\gklammer{Z_k}\times\bigtimes_{n>k} P(A^n)$ is an open neighbourhood of $Z$ and no element of $U$ is a word system. This shows that $\sW(A)\setminus\WS(A)$ is open, hence, $\WS(A)$ is closed.
\end{proof}

\subsection*{`Thinning out' Cartesian systems}

We present some results how to select from a Cartesian system a subsequence and turn that subsequence again into a Cartesian system. We know that by Corollary~\ref{CSisoWScor}, every Cartesian system is isomorphic to a word system, and all results about cardinality sequences also apply to word systems. But it would very cumbersome, indeed, if we had to turn these fresh Cartesian systems into word systems, explicitly. These results are, therefore, instances that illustrate how powerful the considerably more flexible notion of Cartesian system can be as compared with the more restrictive notion of word system.

Let us start with the following triviality---and imagine how notationally complicated it would be to prove it, using only word systems.

\begin{theorem}\label{(dn)thm}
  Let $X^\btimes$ be a Cartesian system with injections $i_{m,n}$ and fix $k\in\N$. Then the family $Y^{\btimes}=\bfam{Y_n}_{n\in\N_0}$ with $Y_n:=X_{nk}$ and with the injections $j_{m,n}:=i_{mk,nk} $ is a Cartesian system.
\end{theorem}

This theorem holds, likewise, for subproduct systems. The next result relies on the important property that, unlike for tensor products, in a Cartesian product of sets there are \it{canonical} projections onto the factors; see Proposition~\ref{prop:word_factorization}. For all sets $S_1$ and $S_2$, define $P_i\colon S_1\times S_2\rightarrow S_i$ by $P_i(s_1,s_2)=s_i$.

\begin{theorem}\label{(n+k)thm}
  Let $X^\btimes$ be a Cartesian system with injections $i_{m,n}$ and fix $k\in\N$. Then the family
  $Y^{\btimes}=\bfam{Y_n}_{n\in\mathbb{N}}$ with
  \begin{equation}
    Y_n=
    \begin{cases}
      X_{n+k},&n>0\\
      \gklammer{\Lambda},&n=0
    \end{cases}
  \end{equation}
together with the injections
  \begin{equation}\label{eq:2}
    j_{m,n}:=(P_{1}\circ i_{m+k,n},P_{2}\circ i_{m,k+n})
  \end{equation}
for $m,n\ge1$ and $j_{m,0}$ and $j_{0,n}$ being (necessarily) the canonical injections is a Cartesian system.
\end{theorem}

\begin{proof}
Note that the construction `commutes' with isomorphisms $\alpha^\btimes\colon X^\btimes\rightarrow X'^\btimes$. (Indeed, since $i'_{m,n}\circ\alpha_{m+n}=(\alpha_m\times\alpha_n)\circ i_{m,n}$, we find
\bmun{
(P_{1}\circ i'_{m+k,n},P_{2}\circ i'_{m,k+n})\circ\alpha_{m+k+n}
~=~
(P_{1}\circ(\alpha_{m+k}\times\alpha_n)\circ i_{m+k,n},P_{2}\circ(\alpha_m\times\alpha_{k+n})\circ i_{m,k+n})
\\
~~~~~~
~=~
(\alpha_{m+k}\circ P_{1}\circ i_{m+k,n},\alpha_{k+n}\circ P_{2}\circ i_{m,k+n})
~=~
(\alpha_{m+k}\times\alpha_{k+n})\circ(P_{1}\circ i_{m+k,n},P_{2}\circ i_{m,k+n}),
}\emun
so that $j'_{m,n}\circ\alpha_{m+k+n}=(\alpha_{m+k}\times\alpha_{k+n})\circ j_{m,n}$.) By Corollary~\ref{CSisoWScor}, every Cartesian system is isomorphic to a word system. We, therefore, may assume that $X^\btimes$ is a word system.

  For a word system $X^{\btimes}$, the definition in \eqref{eq:2} leads to 
  \begin{equation}\label{eq:3}
    j_{m,n}(a_1,\ldots, a_{m+n+k}):=\bigl((a_1,\ldots, a_{m+k}),(a_{m+1},\ldots,a_{m+n+k})\bigr)
  \end{equation}
  for all $m,n>0$. In other words, the $k$ letters `in the middle' $a_{m+1},\cdots, a_{m+k}$ are `replicated' once to the right part of the left factor and once to the left part of the right factor. The maps $j_{m,n}:X_{m+n+k}\to X_{m+k}\times X_{n+k}$ defined in \eqref{eq:3} are clearly injective. The computation
    \begin{align*}
    &\quad((j_{m,n}\times \id_{Y_\ell})\circ j_{m+n,\ell})(a_1,\ldots, a_{m+n+\ell+k})\\
    &=(j_{m,n}\times \id_{Y_\ell})\bigl((a_1,\ldots, a_{m+n+k}),(a_{m+n+1},\ldots, a_{m+n+\ell+k})\bigr)\\
    &=\bigl((a_1,\ldots, a_{m+k}),(a_{m+1},\ldots ,a_{m+n+k}),(a_{m+n+1},\ldots, a_{m+n+\ell+k})\bigr)\\
    &=(\id_{Y_m}\times j_{n,\ell})\bigl((a_1,\ldots, a_{m+k}),(a_{m+1},\ldots, a_{m+n+\ell+k})\bigr)\\
    &=((\id_{Y_m}\times j_{n,\ell})\circ j_{m,n+\ell})(a_1,\ldots, a_{m+n+\ell+k})
  \end{align*}
  proves associativity for $m,n,\ell\ge1$. For the cases involving $m=0$ or $n=0$ or $\ell=0$ there is nothing to prove. So the $Y_{n}=X_{n+k}$ together with the maps $j_{m,n}$ form a Cartesian system. 
\end{proof}

\begin{corollary}\label{cor:shift-condition}
  Suppose for $n\in\N$ there is a function $f\colon \N_0^{n-1}\rightarrow\N_0$ such that for every word system $X^\btimes$ we have
  \begin{equation*}
    \# X_{n}\leq f(\# X_1,\# X_2,\dots,\# X_{n-1}).
  \end{equation*}
Then for every Cartesian system $X^\btimes$ we have
  \begin{equation*}
    \# X_{na+b}\leq f(\# X_{a+b},\# X_{2a+b},\dots,\# X_{(n-1)a+b}).
  \end{equation*}
\end{corollary}

\begin{proof}
  By the preceding two theorems the $Y_n=X_{na+b}$ form a Cartesian system. 
\end{proof}

\begin{corollary}
For every dimension sequence $d_n=\#X_n$ of a word system $X^\btimes$, we have
  \begin{equation}\label{eq:13}
    d_{m+n+k}\leq d_{m+k}d_{n+k}
  \end{equation}
  for all $m,n,k\in\mathbb{N}$. In particular, $d_{k+1}\leq d_k^2$ for every $k>0$. 
\end{corollary}

\begin{proof}
Equation \eqref{eq:13} follows, because the sequence $d_n$ is submultiplicative ($d_{m+n}\le d_md_n$). The formula $d_{k+1}\leq d_k^2$, in the case $k=1$ is directly submultiplicativity, and in the case $k>1$ follows from \eqref{eq:13} with $m=n=1$.
\end{proof}

\begin{corollary}
  Not every submultiplicative sequence $d_n$ is the cardinality sequence of a word system.
\end{corollary}

\begin{proof}
A cardinality sequence fulfills $d_3\leq d_2^2$. However, the sequence $d_1=2$,
$d_2=1$, $d_3=2$, and $d_k=0$ for $k>3$ is submultiplicative, but $d_3\nleq
  d_2^2$.
\end{proof}

Of course, from Corollary~\ref{CSisoWScor} it follows that all corollaries remain true for Cartesian systems, and from Theorem~\ref{sec:cart-syst-word:main-theorem} it follows that all corollaries remain true for subproduct systems replacing cardinality sequence with dimension sequence.

\subsection*{A sufficient criterion motivated by submultiplicativity}

It is well known that for every submultiplicative sequence $(d_n)_{n\in\N_0}$ of nonnegative integers we have $\lim_{n\to\infty} \sqrt[n]{d_n}=\inf_n\sqrt[n]{d_n}$. On the other hand, if we assume the limit is approached monotonously, that is, if we assume
$\sqrt[m+1]{d_{m+1}}\leq\sqrt[m]{d_{m}}$ for all $m\in\N_0$, from
\begin{equation*}
  d_{m+n}=\sqrt[m+n]{d_{m+n}}^{m}\sqrt[m+n]{d_{m+n}}^{n}\leq\sqrt[m]{d_m}^m\sqrt[n]{d_n}^n=d_m d_n
\end{equation*}
we get that $d_n$ is submultiplicative. We may ask, if the condition $\sqrt[m+1]{d_{m+1}}\leq\sqrt[m]{d_{m}}$ for all $m\in\N_0$ is sufficient to be the cardinality sequence of a word system. It turns out that this condition is neither sufficient (Example~\ref{sec:subpr-syst-cart:example_keine_18}) nor necessary (Example~\ref{sec:subpr-syst-cart:example_rand}). However, we may modify the condition to make it at least sufficient.

For $a\in\R$, denote $\left\lceil a\right\rceil:=\min\gklammer{n\in
    \mathbb{Z}\mid n\geq a}$ and $\left\lfloor a\right\rfloor:=\max\gklammer{n\in
    \mathbb{Z}\mid n\leq a}$.

\begin{theorem}\label{thm:suff-crit-motiv}
  Let $(d_n)_{n\in\N_0}$ be a sequence of nonnegative integers such that $d_0=1$ and
  \begin{equation}
    \label{eq:10}
    \left\lceil \sqrt[m+1]{d_{m+1}}\,\right\rceil\leq \left\lfloor\sqrt[m]{d_{m}}\right\rfloor
  \end{equation}
for all $m\ge1$.  Then their exists a word system $X^\btimes$ with $\#X_n=d_n$ for all $n\in\mathbb{N}$.
\end{theorem}

\begin{proof}
Set $X_0=\gklammer{\Lambda}$, $X_1=\gklammer{1,\dots, d_1}$, and choose arbitrary $X_n\subset X_1^n$ such that
 \begin{equation}
   \label{eq:4}
   \gklammer{1,\dots,\left\lfloor\sqrt[n]{d_{n}}\right\rfloor}^n\subset X_n\subset \gklammer{1,\dots, \left\lceil \sqrt[n]{d_{n}}\,\right\rceil}^n.
 \end{equation}
Since $\left\lfloor\sqrt[n]{d_{n}}\right\rfloor^n\leq d_n\leq\left\lceil\sqrt[n]{d_{n}}\right\rceil^n$, this is always possible. We find
\begin{align*}
  X_{m+n}&\subset \gklammer{1,\dots, \left\lceil \sqrt[m+n]{d_{m+n}}\,\right\rceil}^{m+n}\\
  &=\gklammer{1,\dots, \left\lceil \sqrt[m+n]{d_{m+n}}\,\right\rceil}^{m}\times \gklammer{1,\dots, \left\lceil \sqrt[m+n]{d_{m+n}}\,\right\rceil}^{n}\\
  &\subset\gklammer{1,\dots, \left\lfloor \sqrt[m]{d_{m}}\right\rfloor}^{m}\times\gklammer{1,\dots, \left\lfloor \sqrt[n]{d_{n}}\right\rfloor}^{n}\\
  &\subset X_m\times X_n
\end{align*}
for all $m,n\in\N_0$. So, by Theorem~\ref{WS=CSSFthm}, the $X_n$ form a word system over $X_1$.
\end{proof}

By the same Example~\ref{sec:subpr-syst-cart:example_rand}, this sufficient condition is not necessary.

\SQ{\lf  
  \begin{note}\label{note:dimension-rational-time}
    In \cite[Section 4.2]{G15}, our results on thinning out Cartesian systems have been combined with the Balogh Bollob{\'a}s estimate \eqref{eq:BaBo} to show that a function $d\colon \mathbb{Q}_+\to \mathbb{N}_0$ is the \emph{dimension function} of a rational times subproduct system, that is, $d(t)=\dim H_t$ for some subproduct system $H^\botimes=\bfam{H_t}_{t\in\mathbb{Q}_+}$, if and only if 
    \[d(s)\leq \left\lceil\frac{d(t)+1}{2}\right\rceil\left\lfloor\frac{d(t)+1}{2}\right\rfloor\]
    holds for all $s>t$. Analogous results hold for the cardinality functions of rational time Cartesian systems and for the cardinality functions of continuous time Cartesian systems. The problem of characterizing dimension functions of continuous time subproduct systems is not yet solved, and the solution might depend on additional continuity assumptions.
  \end{note}
}


\section{(Counter)examples with word systems of graphs}
\label{sec:count-with-word}

In this section, we establish a connection between word systems and directed graphs. In fact, every directed graph without multiple edges is a word system. Moreover, every word system is a subsystem of such a \it{graph system}. As application, we provide examples that show that $\sqrt[n+1]{d_{n+1}}\leq\sqrt[n]{d_{n}}$ for all $n\in\N_0$ is neither necessary nor sufficient for the existence of a word system with cardinality sequence $d_n$. Of course, this implies that the sufficient condition in Theorem~\ref{thm:suff-crit-motiv}, which is even stronger, is not necessary.

By a \hl{graph}, we will always mean a directed graph, possibly with loops, but without multiple edges. That is, a graph is a pair $(V,E)$, where $V$ is a set, whose elements are called \hl{vertices}, and a subset $E$ of $V\times V$, whose elements are called \hl{edges}.\pagebreak[3]

\begin{theorem}{~}
\begin{enumerate}
\item
  Let $\Gamma=(V,E)$  be a graph and set
  \begin{align*}
    E_0&:=\gklammer{\Lambda}\\
    E_1&:=V\\
    E_2&:=E\\
    E_n&:=\gklammer{(v_1,\ldots, v_n)\in V^n\mid (v_i,v_{i+1})\in E~\forall i=1,\ldots, n-1}.
  \end{align*}
(That is, $E_n$ consists of all \hl{paths} of length $n-1$.) 

Then $E_n=X_n((V\times V)\setminus E)$. In particular, the $E_n$ form a word system over $V$, the \hl{graph system} $X_\Gamma^\btimes$.

\item
Every word system $X^\btimes$ is  a subsystem of the graph system $X^\btimes_{(X_1,X_2)}$ associated with the graph $(X_1,X_2)$.
\end{enumerate}
\end{theorem}

\begin{proof}{~}

\begin{enumerate}
\item
By definition,
\[
X_n((V\times V)\setminus E)
:=
\gklammer{(v_1,\ldots, v_n)\in V^n\mid (v_i,v_{i+1})\notin(V\times V)\setminus E~\forall i=1,\ldots, n-1}=E_n.
\]
\item
For each word $(v_1,\ldots,v_n)\in X_n$, the $(v_i,v_{i+1})$ are subwords, hence belong to $X_2$. So $X_n$ is a subset of $E_n$.
\end{enumerate}
\end{proof}
\noindent
Of course, $(V\times V)\setminus E$ is reduced. So graph systems are precisely those word systems which have a reduced set of excluded words consisting only of words of length $2$.

Let $\Gamma=(V,E)$ be a graph  with
$V=\gklammer{1,\ldots,d}$. Denote by
\begin{equation*}
A_{ij}=
\begin{cases}
  1&\text{for $(i,j)\in E$,} \\
  0&\text{else}
\end{cases}
\end{equation*}
in $M_{d}$ its \hl{adjacency matrix}.
Then, obviously, the number of
paths  of length $n$ from $i$ to $j$ is given by the $i$-$j$-entry of
$A^n$. 
For any matrix $B\in M_{d}$ denote by
\begin{equation*}
   S (B):=\sum_{i,j=1,\ldots,d}B_{i,j}
\end{equation*}
 the sum of all its entries. So $S(A^{n-1})=\# E_n$.
Denote by  $\Eins_{m\times n}$ the $m\times n$ matrix with all entries equal
to $1$ and put $\Eins_d:=\Eins_{d\times 1}\in M_{d,1}=\C^d$. Note that in this notation $\Eins_{d\times d}\Eins_d=d\Eins_d$.

\begin{example}\label{sec:subpr-syst-cart:example_rand}
  Let $(V,E)$ be the graph with $d+1$ vertices and adjacency matrix $A=
  \sMatrix{
    0&\Eins_d^t\\ 
    \Eins_d&0
  }$. 
  Then $\# E_2= S (A)=2d$. Since $A^2=
\sMatrix{
  d&0\\
  0&\Eins_{d\times d}
}
$ we have $\# E_3= S(A^2)=d^2+d$.
\end{example}

\begin{example}\label{sec:subpr-syst-cart:example_wenig_Kanten}

Let $(V,E)$ be a graph with $\#E=1$, so $E=\gklammer{(v,w)}$. If $v\neq w$, then for $n>2$ there is no path of length $n-1$, so $E_n=\emptyset$. If $v=w$, then $E_n=\gklammer{(v,v,\ldots,v)}$, so, $\#E_n=1$ .

Let $(V,E)$ be a graph with $\#E=2$, that is, its adjacency matrix $A$ is the sum of two distinct matrix units $E_{ij}$ and $E_{kl}$. We find
\begin{equation*} \#E_3= S(A^2)= S((E_{ij}+E_{kl})^2)=\delta_{ij}+\delta_{il}+\delta_{kj}+\delta_{kl}.
\end{equation*}
Since three of the equalities $i=j,i=l,k=j$ and $k=l$ necessarily lead to $i=j=k=l$, necessarily $\#E_3\leq 2$. Since every word system is a subsystem of its graph system, the implication $\#X_2\leq2\Rightarrow \#X_3\leq 2$ holds for \it{all} word systems. In other words, if we define
\begin{equation*}
f(d_1,d_2):=
\begin{cases}
  2&d_2\leq 2\\
  d_1^3&\text{otherwise},
\end{cases}
\end{equation*}
then $d_3\leq f(d_1,d_2)$ for all cardinality sequences of word systems. By Corollary~\ref{cor:shift-condition}, \[d_4=d_{3+1}\leq f(d_{1+1},d_{2+1})=f(d_2,d_3)=2\] because $d_3\leq 2$, and so forth. Hence $\#E_2\leq 2$ implies $\#E_{2+k}\leq 2$ for all $k\in\N_0$.
\end{example}

\begin{observation}
  A straightforward calculation gives
  \begin{align*}
      S(\Eins_{d\times d}A)= S(A\Eins_{d\times d})=d S(A)
  \end{align*}
  for all $A\in M_d$.
  Put $\overline{A}:=\Eins_{d\times d}-A$. Combining the two equations
  \begin{align*}
     S(\overline{A}A)&= S((\Eins_{d\times d}-A)A)= S(\Eins_{d\times d}A-A^2)=d S(A)- S(A^2)\\
  \intertext{and}
     S(\overline{A}A)&= S(\overline{A}(\Eins_{d\times d}-\overline{A}))= S(\overline{A}\Eins_{d\times d}-\overline{A}^2)=d S(\overline{A})- S(\overline{A}^2),
  \end{align*}
  we get
  \begin{equation}\label{eq:6}
    S(A^2)= S(\overline{A}^2)+d( S(A)- S(\overline{A})).
  \end{equation}
\end{observation}

\begin{example}\label{sec:subpr-syst-cart:example_keine_18}
  Let $\Gamma=(V,E)$ be a graph with $3$ vertices and $7$ edges. For its adjacency matrix $A$ we, thus, have $ S(A)=7$ and $ S(\overline{A})=2$. Note that $\overline A$ is the adjacency matrix of the \hl{complementary graph} $\overline{\Gamma}:= (V,(V\times V)\setminus E)$. Therefore, by Example~\ref{sec:subpr-syst-cart:example_wenig_Kanten}, we obtain $ S(\overline{A}^2)\leq2$. So \eqref{eq:6} yields
\begin{equation*}
  \# E_3
  = S(A^2)
  = S(\overline{A}^2)+d( S(A)- S(\overline{A}))
  \leq 2+3(7-2)=17.
\end{equation*}
This shows that a graph with $3$ vertices and $7$ edges has at most $17$ paths of length two.
\end{example}

\begin{example}
  In a graph with $d$ vertices and $d^2-1$ edges we have $ S(A)=d^2-1$, $ S(\overline{A})=1$ and $ S(\overline{A}^2)$ is either $1$ or $0$, depending on whether the missing edge is a loop or not. Using again \eqref{eq:6}, we find
  \begin{align*}
    \# E_3= S(\overline{A}^2)+d(d^2-2)=
    \begin{cases}
      d^3-2d&\text{if the missing edge is a loop,}\\
      d^3-2d+1&\text{if the missing edge is not a loop.}
    \end{cases}
  \end{align*}
\end{example}

We learn from these examples that the condition $\sqrt[m+1]{d_{m+1}}\leq\sqrt[m]{d_{m}}$ is neither
  sufficient nor necessary for the existence of a Cartesian system
  $X^{\btimes}$ with $\# X_n=d_n$. By Example~\ref{sec:subpr-syst-cart:example_keine_18} there is no system with
  $\#X_1=3, \#X_2=7, \#X_3=18$. But  $18^2=324<343=7^3$. So, the sequence $d_1=3,d_2=7,d_3=18,d_n=0$ for $n>3$ fulfills the condition. So the condition is not sufficient. In Example~\ref{sec:subpr-syst-cart:example_rand}, putting $d=10$, we get a system with $\#X_1=11,
  \#X_2=20, \#X_3=110$. But $110^2>10000>8000=20^3$. So the condition is not necessary. 
As mentioned in the beginning of this section, this implies that the stronger condition \eqref{eq:10} is not necessary either. 

Especially in view of Theorem~\ref{theo:local_to_global}, the following class of questions is interesting: Fixing (some of) the cardinalities $\#X_1,\ldots,\#X_n$, what is the maximal possibility for $\#X_{n+1}$ in a word system $X^\btimes$? The question, which graph with $d_1$ vertices and $d_2$ edges has the maximal number of paths of length $2$, is clearly of the above type with $n=2$. It was first investigated by Katz in \cite{Kat71}, who gave an answer only for special values of $d_1$ and $d_2$. A complete answer was given by Aharoni in \cite{Aha80}  by exhibiting four special types of graphs, (two of them are close to being complete graphs, two of them are close to being complements of complete graphs) one of which is maximal for any choice of $d_1$ and $d_2$. This allows one to determine the maximal $d_3$ such that there is a word system $X^\btimes$ with $\#X_1=d_1,\#X_2=d_2$ and $\#X_3=d_3$. Similar results for undirected graphs can be found in \cite{AhKa78},\cite{PPS99} and \cite{AFNW09}. It seems, the questions for higher $n$ are still open problems.

\lf
We close by briefly mentioning a relation to operator algebras which we do not address here, but which promises to deepen the connection between subproduct systems and graphs. With a subproduct system $H^\botimes=\bfam{H_n}_{n\in\N_0}$ we can associate its \hl{Fock space} $\sF(H^\botimes):=\bigoplus_{n\in\N_0}H_n$. For each $x\in H_1$, we define the \hl{creation operator} by $\ell^*(x)x_n=xx_n$. Apart from non-selfadjoint operator algebras, Davidson, Ramsey, and Shalit analyzed the \nbd{C^*}algebras generated by $\ell^*(x)$ (the so-called \it{Toeplitz algebras}) and certain \it{universal} \nbd{C^*}algebras (the so-called \it{Cuntz algebras}) for subproduct systems. The Fock space $\sF(H^\botimes)$ is a special instance of a so-called \it{interacting Fock space} (Accardi, Lu, and Volovich \cite{ALV97}), and since Accardi and Skeide \cite{AcSk00a} it is known that the Toeplitz algebras of interacting Fock spaces are subalgebras of the Pimsner-Toeplitz algebra on a suitable full Fock module (Pimsner \cite{Pim97}). On the other hand, graphs are associated with \it{graph \nbd{C^*}algebras} that can be viewed as quotients of certain Pimsner-Toeplitz algebras. One may show that the Pimsner-Toeplitz algebra of a finite graph with no double edges coincides with the Pimsner-Toeplitz algebra of the subproduct system associated with that graph. This result and the relation of the Pimsner-Toeplitz algebras (and also the associated universal \it{Cuntz-Pimsner algebras}) of subproduct systems of general word systems with those of the containing graph are discussed in Gerhold and Skeide \cite{GeSk13p}.

\lf\lf\lf\noindent
\bf{Acknowledgments.}
We gratefully acknowledge the support of our Departments for supporting a number of mutual visits. We also wish to thank Roland Speicher for inviting us to Saarbrücken and for useful discussions. M.~Gerhold acknowledges funding from the German Research Foundation (DFG), project no.\,397960675.

\newcommand{\Swap}[2]{#2#1}\newcommand{\Sort}[1]{}\def\cprime{$'$}

\linespread{1}
\setlength{\parindent}{0pt}

\end{document}